\documentclass[a4paper,12pt,reqno]{amsart}

\usepackage[english]{babel}
\usepackage[T1]{fontenc}
\usepackage[utf8]{inputenc}

\usepackage[nobysame,initials]{amsrefs}
\usepackage{bbm}
\usepackage{enumitem}
\usepackage[top=3cm,
  bottom=3cm,
  left=2cm,
  right=2cm]{geometry}
\usepackage{graphicx}
\usepackage[hidelinks]{hyperref}
\usepackage{setspace}
\usepackage{upgreek}
\usepackage{url}
\usepackage{float}

\DefineSimpleKey{bib}{primaryclass}{}
\DefineSimpleKey{bib}{archiveprefix}{}

\BibSpec{arXiv}{
+{}{\PrintAuthors}{author}
+{,}{ \textit}{title}
+{}{ \parenthesize}{date}
+{,}{ arXiv preprint }{eprint}
+{,}{ primary class }{primaryclass}
+{.}{}{transition}
}

\usepackage{amsmath}
\usepackage{amssymb}
\usepackage{amsthm}
\usepackage{mathtools}
\usepackage{mismath}
\usepackage{nicematrix}

\renewcommand{\R}{\mathbb R}
\newcommand{\mbf}[1]{\mathbf{#1}}
\newcommand{\one}[1]{\mathbf{1}_{#1}}
\newcommand{\ind}[1]{\mathbbm{1}_{#1}}

\newcommand{\dZ}[1]{\mathrm{Z}\left(#1\right)}
\newcommand{\Vol}[2]{\mathrm{Vol}_{#1}\left(#2\right)}
\newcommand{\Laplace}[2]{\mathcal{L}\left[#1\right]\left(#2\right)}
\renewcommand{\Res}[2]{\mathrm{Res}\left(#1;#2\right)}

\newtheorem{theorem}{Theorem}
\newtheorem{proposition}[theorem]{Proposition}
\newtheorem{lemma}[theorem]{Lemma}

\numberwithin{equation}{section}

\title{Minimal central slices of the regular simplex} 

\author{Gergely Ambrus and Barnabás Gárgyán}

\date{\today}

\begin{document}

\begin{abstract}
We prove that minimal-volume hyperplane sections of the regular simplex through its centroid are parallel to a facet. The proof combines variational methods with Fourier-analytic techniques and zero-diminishing arguments to show that every critical normal vector has at most three distinct non-zero coordinates. Analysis of the two- and three-value cases then yields the sharp lower bound.
\end{abstract}

\maketitle

\section{Introduction}\label{sec:intro}

An $n$-dimensional regular simplex is the convex hull of $n+1$ affinely independent points in $\R^{n+1}$ whose pairwise distances are all equal. The canonical example is the {\em standard regular simplex} $\Delta_n$, defined as the convex hull of the standard basis vectors $\{\mbf e_i\}_{i=1}^{n+1}$ in $\R^{n+1}$. We consider its {\em central hyperplane sections}, that is, sets of the form $H\cap \Delta_n$, where $H\subset\R^{n+1}$ is a hyperplane passing through the centroid $\frac1{n+1}\sum_{i=1}^{n+1}\mbf e_i$ of $\Delta_n$. In the present paper, we resolve the long-standing problem of determining the smallest possible $(n-1)$-dimensional volume of such a section.

We represent each central hyperplane as $H=\mbf a^\perp$, where $\mbf a = (a_1, \ldots, a_{n+1})\in\R^{n+1}$ is a unit vector perpendicular to $(1,1,\ldots,1)$. The latter condition ensures that $H$ contains the centroid of $\Delta_n$; see Figure~\ref{fig:simplex}. Throughout this work, $\mbf a$ will denote such a vector.

\begin{figure}[h]
\centering
\includegraphics[scale=.8]{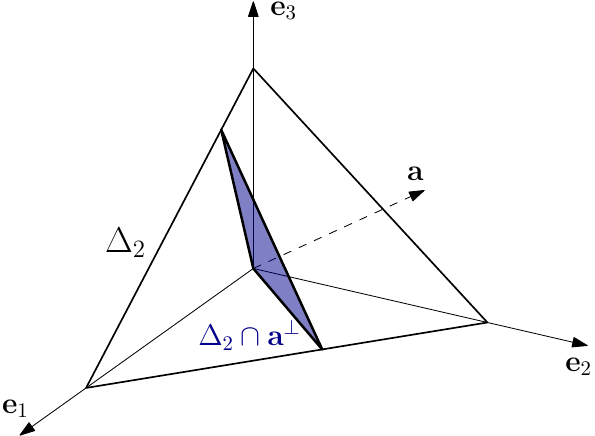}
\caption{Embedding $\Delta_2$ into $\R^3$ and slicing it with the plane $\mbf a^\perp$.}
\label{fig:simplex}
\end{figure}

The \emph{simplex section function} is defined by
\begin{equation}\label{eq:Vsimplex}
V(\mbf a):=\Vol{n-1}{\Delta_n\cap \mbf a^\perp}.
\end{equation}
Our aim is to minimise \eqref{eq:Vsimplex} over all admissible vectors $\mbf a$.

It has long been conjectured that the minimal central hyperplane sections are precisely those parallel to a facet of $\Delta_n$, with corresponding volume
\[
V_{\mathrm{min}}
=
\frac{\sqrt{n+1}}{(n-1)!}
\Big(\frac{n}{n+1}\Big)^{n-\frac12}.
\]
The normal vectors of these facet-parallel sections are obtained from the vector
\[
\mbf a_{\min}:=
\bigg(
\sqrt{\frac{n}{n+1}},
-\frac{1}{\sqrt{n(n+1)}},
\ldots,
-\frac{1}{\sqrt{n(n+1)}}
\bigg)
\]
by permuting the coordinates and multiplying by $-1$.

\medskip
The main result of this paper confirms the conjecture.

\begin{theorem}\label{th:simplex}
For every unit vector $\mbf a\in\R^{n+1}$ with $\sum_{j=1}^{n+1}a_j=0$, we have
\begin{equation}\label{eq:conj}
V(\mbf a)\geq V_{\mathrm{min}}.
\end{equation}
Equality holds if and only if $\mbf a$ is obtained, up to permutation of coordinates and multiplication by $-1$, from the vector $\mbf a_{\min}$.
\end{theorem}

\subsection{Terminology and history}
The study of extremal sections of convex bodies is a well-established area of convex geometry. Among the most prominent examples are sections of the cube, the cross-polytope, and $\ell_p$-balls. For related results, see the recent comprehensive survey~\cite{TN-survey}. A standard approach to studying volumes of hyperplane sections combines probabilistic methods with Fourier-analytic techniques. For symmetric convex bodies, this leads to representations of the section volume by real integrals, which can then be analysed using various methods.

The simplex, however, is not centrally symmetric, and the same approach naturally leads to an integral representation involving a complex-valued function. This makes the analysis of extremal sections, and in particular of minimal sections, considerably more delicate, as we shall see below. We first recall the corresponding integral representation for~\eqref{eq:Vsimplex}, following Webb~\cites{Webb,WebbD}.

\smallskip

For a continuous random variable $Y$, we will denote its probability density function by~$f_Y$. Throughout the paper, $X$ will denote an exponentially distributed random variable with parameter $1$, whose probability density function is 
\[
f_X(x) = \e^{-x}\ind{x\geq0}\,,
\]
where $\ind{A}$ denotes the indicator function of the set $A$. The variables $X_1,\ldots,X_{n+1}$ will always be independent copies of $X$. We define $F:\R^{n+1}\to[0,\infty[$ as
\[F(\mbf x)=\prod_{j=1}^{n+1}f_{X_j}(x_j).\]
Then, a change of variables yields
\begin{align*}
	\int_{\mbf a^\perp}F(\mbf x)\di\mbf x&=
	\int_{\mbf a^\perp}\prod_{j=1}^{n+1}f(x_j)\di\mbf x=\\[5pt]&=
	\int_0^\infty \e^{-t\sqrt{n+1}}\,\Vol{n-1}{(t\sqrt{n+1})\Delta_n\cap\mbf a^\perp}\di t=\\[5pt]&=
	\int_0^\infty \e^{-t\sqrt{n+1}}\big(t\sqrt{n+1}\big)^{n-1}\,\Vol{n-1}{\Delta_n\cap\mbf a^\perp}\di t=\\[5pt]&=
	\frac{(n-1)!}{\sqrt{n+1}}\,\Vol{n-1}{\Delta_n\cap\mbf a^\perp},
\end{align*}
accordingly,
\begin{equation}\label{eq:Va-int}
V(\mbf a)=\frac{\sqrt{n+1}}{(n-1)!}\int_{\mbf a^\perp}F(\mbf x)\di\mbf x.
\end{equation}
Notice that the integral on the right-hand side is the probability density function of $\sum_{j=1}^{n+1}a_jX_j$ at zero. Therefore, using characteristic functions (that is, Fourier transform), this integral can be expressed as
\begin{equation}\label{eq:G0}
 \sigma(\mbf a) : = \frac1{2\pi}\int_{-\infty}^\infty\prod_{j=1}^{n+1}\frac{1}{1+\i a_jt}\di t
\end{equation}
where $\i$ stands for the imaginary unit. We note that the integral exists because  $\mbf a$ has at least two nonzero coordinates. Then, by \eqref{eq:Va-int},
\begin{equation}\label{eq:Volume-sigma}
V(\mbf a)=\frac{\sqrt{n+1}}{(n-1)!}\,\sigma(\mbf a).
\end{equation}

Bounds for the volume of central sections of the simplex thus can be achieved by maximising $\sigma(\mbf a)$ subject to the constraints  $\sum_{j=1}^{n+1} a_j = 0$ and  $\sum_{j=1}^{n+1} a_j^2 = 1$. The global maximum was determined by Webb~\cite{Webb} in 1996, who provided two different arguments to show that $\sigma(\mbf a)$ is at most $\frac1{\sqrt2}$, and that this bound is tight. This implies that central sections that contain $n-1$ vertices of $\Delta_n$ have maximal volume, see \cite{Webb}*{Theorem 4}. Recently, Brazitikos and Pandis~\cite{BPg26}*{Theorem 3.1} proved that all local maxima have exactly one negative coordinate, up to changing the orientation of $\mbf a$. Consequently, they gave an alternative proof for determining the global maxima, and also obtained the global minima subject to this strong sign constraint.

The first lower bound on $V(\mbf a)$ was established in 2013 by Brzezinski~\cite{Brz}*{Theorem 1.1}, who proved that 
\[V(\mbf a) \geq \frac1{2\sqrt3}\Big(\frac{{n+1}}{n}\Big)^{n-\frac12} \, V_{\min}.\]
As $n  \to \infty$, the multiplicative factor converges to $0.7847$.

More recently, Tang~\cite{Tang} made substantial progress towards Theorem~\ref{th:simplex} by establishing an asymptotically sharp lower bound: every central section has volume at least 
\[
\frac1\e\Big(\frac{n+1}{n}\Big)^{n-\frac12} \,V_{\min}.
\]
Note that the multiplicative factor converges to $1$ as $n\to\infty$. His argument is based on transforming $\sigma(\mbf a)$ to a contour integral whose integrand is real-valued and positive along the contour, which allows for estimating it from below. 

In 2017, Dirksen \cites{Dirksen,DirksenD} claimed to verify Theorem~1 in dimensions up to $4$ in \cite{Dirksen}*{Theorem~1.3}. His argument relies crucially on \cite{Dirksen}*{Proposition~3.3}, which asserts that balancing the negative coordinates of the normal vector $\mbf a$ decreases the volume of the corresponding section. There appears, however, to be a gap in the proof of this proposition, stemming from an inequality that does not hold in the required generality\footnote{
The error in~\cite{Dirksen} occurs on p.~2576, line~$-13$, where the factor $\prod_{k=1,k\neq j}^P\frac1{1-\frac{a_k}{a_j}}$ may be negative; consequently, the asserted inequality does not hold in general.}. Thus, as far as we are aware, Proposition~3.3, and consequently the claimed verification of Theorem~1 in dimensions up to $4$, are not established by the argument given there. Numerical calculations nevertheless provide some evidence that the underlying assertion may hold.

Another natural approach to the extremal problem is to characterise the critical sections of $\Delta_n$, namely those whose normal vectors are critical points of the simplex section function~\eqref{eq:Vsimplex}. Filliman~\cite{Filli} proposed that for a critical central section $P^*$, its centroid should coincide with that of $\Delta_n$ and each facet of $P^*$ should be normal to the line joining these centroids. He indicated that these conditions would imply that the minimal two-dimensional sections are equilateral triangles. Webb later showed in his thesis~\cite{WebbD} that the proposed conditions do not hold in full generality: the central section  of $\Delta_3$ with normal vector 
\[\frac 1 4 \big(-1,-1, 1+ (\sqrt{33}+5)^{\frac12},1-(\sqrt{33}+5)^{\frac12} \big)\]
is a critical section that violates these criteria. Had Filliman's centroid conditions held for all critical sections, Theorem~\ref{th:simplex} would have followed rather directly; see \cite{WebbD}*{Chapter~4} for details.  This may explain why Filliman \cite{Filli}*{\S3, Comment~(h)} expected Theorem~1 to hold. Webb's example nevertheless shows that the geometry of critical sections is more subtle, and that their characterisation requires a more refined analysis.

One can also obtain analytic conditions on critical sections by applying the method of Lagrange multipliers to \eqref{eq:Vsimplex}, viewed as a function of the coordinates $\{a_j\}_{j=1}^{n+1}$ subject to the conditions $\sum_{j=1}^{n+1}a_j=0$ and $\sum_{j=1}^{n+1}a_j^2=1$. This approach serves as a fundamental tool in studying central as well as non-central sections of the the simplex.  Regarding the latter, sections of globally maximal volume that are close to the centroid of the simplex were described by Brazitikos and Pandis~\cite{BPg26} while sections relatively far from the centroid by König~\cites{König21,König23}, who also investigated sections of locally extremal volume at certain distances. Specifically, he showed that $\mbf a_{\min}$ is \emph{local} minimiser of $V(\mbf a)$; see \cite{König21}*{Proposition 1.3 \emph{(b)}}. In the case of $\Delta_2$ and $\Delta_3$, he determined all extremal hyperplane sections at a fixed distance from the centroid; see \cite{König23}*{Propositions~1.2 and~1.3}. In particular, these results verify Theorem~\ref{th:simplex} in the corresponding dimensions.

It is also possible to evaluate the integral~\eqref{eq:G0} explicitly. When the coordinates of $\mbf a$ are pairwise distinct and nonzero, Webb~\cite{Webb} obtained the explicit expression
\begin{equation}\label{eq:interpol}
\sigma(\mbf a) = \frac12\sum_{k=1}^{n+1}\frac1{\abs{a_k}}\prod_{j\neq k}\frac{a_k}{a_k-a_j}
\end{equation}
by using contour integration. Repeated or zero coordinates can be handled by first combining the corresponding terms and then passing to the appropriate limit. This makes the explicit formula increasingly cumbersome in higher dimensions. For illustrations of this approach, see the works of König~\cites{König21,König23}.

The study of volumes of simplex sections is equivalent to estimating the density function of weighted sums of independent exponential random variables at a prescribed location. Thus, the problem studied here has strong probabilistic connections. For an illustration of this scheme, see~\cite{Tang}, and~\cite{MRTT} that derives estimates for moments of log-concave random variables motivated by Webb's result on maximal sections. We also note that, similarly  to the present approach, Brazitikos and Pandis \cite{BPg25} proved Hunter's conjecture by using a probabilistic representation and reducing the study to extrema with at most three coordinate values.

\subsection{Ingredients of the proof}
        
The proof of Theorem~\ref{th:simplex} applies the method of Lagrange multipliers directly  to the \emph{integral formula}~\eqref{eq:G0}, rather than to \eqref{eq:interpol}. Its crux is the following theorem, established in Section~\ref{sec:crits} using zero-diminishing properties. From now on, {\em section} will always refer to a central hyperplane section of the simplex. 

\addvspace{\medskipamount}
\begin{theorem}\label{th:critic}
The normal vector of any critical section of $\Delta_n$ has at most three distinct nonzero coordinates.
\end{theorem}

For a minimising vector with no zero coordinates, Theorem~\ref{th:critic} therefore reduces the problem to the cases of two or three distinct coordinate values. Vectors with zero coordinates will be handled by induction in the proof of Theorem~\ref{th:simplex}. The two remaining cases are treated in Sections~\ref{sec:critic2} and \ref{sec:critic3}, respectively.

\addvspace{\medskipamount}
\begin{proposition}\label{prop:critic2}
Any section of $\Delta_n$ whose normal vector has exactly two distinct values has volume at least $V_{\min}$. Equality holds only for $\mbf a_{\min}$ up to permutation of coordinates and multiplication by $-1$.
\end{proposition}

\addvspace{\smallskipamount}
\begin{proposition}\label{prop:critic3}
Among sections of $\Delta_n$ whose normal vector has three -- not necessarily distinct -- coordinate values whose multiplicities are fixed, the minimisers of the volume either have only two distinct coordinate values, or have a zero coordinate value.
\end{proposition}

Together with an induction on the number of zero coordinates, these statements imply Theorem~\ref{th:simplex}; the proof is completed in Section~\ref{sec:completion}.

\section{A necessary condition for critical sections}\label{sec:crits}

To prove Theorem~\ref{th:critic}, we minimise $\sigma(\mbf a)$ subject to $\sum_{j=1}^{n+1} a_j = 0$ and $\sum_{j=1}^{n+1} a_j^2 =1$ via the integral representation \eqref{eq:G0}. We may assume that $a_1>0$ and $a_2<0$. 

We prove that the constrained critical points of $\sigma(\mbf a)$ have at most three distinct nonzero coordinate values. To that end, we first apply  Lagrange multipliers, and show that each coordinate of a critical vector satisfies a certain equation \eqref{eq:lms}. We then estimate the number of zeros of the resulting equation by means of zero-diminishing properties.

The second part of the proof relies on the probabilistic interpretation of \eqref{eq:G0}. As we saw in Section~\ref{sec:intro}, $\sigma(\mbf a)$ coincides with $f_{\Sigma_{n+1}}(0)$, where 
\begin{equation}\label{eq:sigmadef}
\Sigma_{n+1}:=\sum_{j=1}^{n+1}a_jX_j.
\end{equation}
The density of $\Sigma_{n+1}$ can be expressed as a convolution
\begin{equation}\label{eq:convolution}
    f_{\Sigma_{n+1}}(x) = f_{\Sigma_n} * f_{a_{n+1} X} (x),
\end{equation}
where for $a\neq 0$,
\begin{equation}\label{eq:exp-dens}
f_{a X}(x)=\frac1{\abs a}\e^{-\frac xa}\one{\sgn (a) x\geq0}.
\end{equation}

Consequently, we 
need to study the integral of $f_{\Sigma_{n}}$ against an exponential weight. This is the central idea underlying our approach. The subsequent zero-diminishing properties will play a key role in this process.

Let $I\subset \R$ be some (possibly infinite) interval and $g:I\to\R$ be a continuous function. The \emph{number of distinct zeros} of $g$ will be denoted by~$\dZ{g}$. The \emph{number of sign changes} of $g$ is the supremum, over all finite monotone sequences $(x_i)_i$ in $I$, of the number of sign changes in the corresponding sequence $(g(x_i))_i$, with zero entries ignored. This number will be denoted by $\nu(g)$. If $\nu(g)$ is finite, then $I$ can be divided into $\nu(g)+1$ subintervals such that $g$ is not identically zero on any of them, has a constant sign on each of them, and has opposite signs on neighbouring subintervals.

The first diminishing property is straightforward. We nevertheless record it as a separate lemma, since it serves as a basic tool in the proof of Theorem~\ref{th:critic}.

\begin{lemma}[Zero-diminishing property of the derivative]\label{lem:cumi}
Let $g:\R\to\R$ be a differentiable function. Then, $\dZ{g}\leq\dZ{g'}+1$. Moreover, if $\lim_{x\to-\infty}g(x)=0$, then $\dZ{g}\leq\dZ{g'}$.
\end{lemma}

\begin{proof}
The first assertion follows from Rolle's theorem. The second part is implied by the fact that if $\lim_{x\to-\infty}g(x)=0$ and $g$ has only finitely many zeroes, then $g'$ must also have a zero to the left of the smallest zero of $g$.
\end{proof}

\medskip
For the second diminishing property, we  introduce the notation $\Laplace{g}{x}$ for the Laplace transform of a function $g$ at $x$, that is $\Laplace{g}{x}=\int_0^\infty g(t)\e^{-t x}\di t$. The lemma below states that the number of sign changes of a function bounds the number of real zeros of its Laplace transform. It was first formulated by Laguerre~\cite{Laguerre}*{p. 28}. Later, Pólya~\cite{Pólya} gave a precise proof along with several consequences. Further corollaries may be found in Parodi's works~\cites{Parodi1,Parodi2}. Below, we recover Pólya's argument~from \cite{Pólya} in more detail, with the modernised version of `sign changes'. 

\begin{lemma}[Zero-diminishing property of the Laplace transform, \cite{Pólya}]\label{lem:LT}
Let $g:[0,\infty[\,\to\R$ be a continuous function not identically zero. Assume that its Laplace transform $\Laplace{g}{x}$ is absolutely convergent for $x>x_0$. Then, $\Laplace{g}{x}$ has at most $\nu(g)$ distinct zeros in $]x_0,\infty[$.
\end{lemma}

\begin{proof}

Let us simply denote $\nu(g)$ by $\nu$. If $\nu=\infty$, the statement is trivial. Thus, we assume that $\nu < \infty$, and apply induction on $\nu$.

First, we consider the base case: $\nu=0$. We may assume that $g\geq0$. Then, $\Laplace{g}{x}>0$ for every $x>x_0$, so the transform has no zeros.

Suppose that $\nu\geq1$ and the theorem holds for every integrand with at most $\nu-1$ sign changes. Choose a point $t_0\in\,]0,\infty[$ where $g$ changes sign and define
\[G(x)=\e^{t_0x}\Laplace{g}{x}.\]
The functions $G$ and $\Laplace{g}{.}$ have exactly the same zeros.  
Differentiation under the integral sign yields
\begin{equation*}
G'(x)=\frac{\di}{\di x}\bigg(\e^{t_0x}\int_0^\infty g(t)\e^{-xt}\di t\bigg)=\e^{t_0x}\int_0^\infty(t_0-t)g(t)\e^{-xt}\di t.
\end{equation*}
Since multiplication by the factor $(t_0-t)$ preserves the sign on  $]-\infty, t_0[$ and reverses it on $] t_0, \infty[$, and $g$ has a sign change at $t_0$, the integrand $(t_0-t)g(t)$ above has precisely $\nu-1$ sign changes. Therefore, by the induction hypothesis, $G'$ has at most  $\nu-1$ distinct zeros in $]x_0,\infty[$. By Lemma~\ref{lem:cumi}, this implies that the number of distinct zeroes of $G$, and hence of $\Laplace{g}{.}$, is at most $\nu$. This completes the induction.
\end{proof}

\vspace{\medskipamount}
The proof of Theorem~\ref{th:critic} is divided into three steps: a critical-point equation, a Laplace transform reduction, and a zero-diminishing induction. We start with the variational characterisation.

\begin{lemma}[Critical-point equation]\label{lem:critical-equation}
Let $\mbf a=(a_1,\ldots,a_{n+1})$ be a constrained critical point of $\sigma(\mbf a)$ subject to the conditions $\sum_{j=1}^{n+1} a_j=0$ and $\sum_{j=1}^{n+1} a_j^2=1$. Set 
\[
\varphi(t)=\prod_{j=1}^{n+1}\frac1{1+\i a_jt} \quad \text{and} \quad \tau=\sigma(\mbf a)=\frac1{2\pi}\int_{-\infty}^{\infty}\varphi(t)\di t.
\]
Then there exists $\mu\in\R $, such that every coordinate $a_k$ is a solution of the equation
\begin{equation}\label{eq:lms}
- \tau  x^2 + \mu x + \tau =\frac1{2\pi} \int_{-\infty} ^\infty  \varphi(t)\frac 1 {1 + \i x t} \di t.
\end{equation}
\end{lemma}

\begin{proof}

Since all coordinates $a_j$ are nonzero, the integrand $\varphi(t)$ and its derivatives with respect to the variables $a_j$ have sufficient decay at infinity to justify differentiation of the formula~\eqref{eq:G0} under the integral sign. Let $\mbf a$ be a constrained critical point of $\sigma$. By the Lagrange multiplier method, there exist $\lambda,\mu\in\R$ such that, for every $k$,
\begin{equation}\label{eq:lak}
    \lambda a_k+\mu=-\frac1{2\pi}\int_{-\infty}^{\infty}\varphi(t)\frac{\i t}{1+\i a_kt}\di t.
\end{equation}
Multiplying \eqref{eq:lak} by $a_k$ gives
\begin{equation}\label{eq:lak2}
\lambda a_k^2+\mu a_k=-\frac1{2\pi}\int_{-\infty}^{\infty}\varphi(t)\frac{\i a_kt}{1+\i a_kt}\di t.
\end{equation}
Since $\sum_{k=1}^{n+1} a_k =0$ and $\sum_{k=1}^{n+1} a_k^2 = 1$, summing over $k$ leads to
\begin{equation*}\label{eq:lambda}\lambda=-\frac1{2\pi}\int_{-\infty}^\infty\varphi(t)\sum_{k=1}^{n+1}\frac{\i a_k t}{1+\i a_kt}\di t=\frac1{2\pi}\int_{-\infty}^\infty\varphi'(t)t\di t=-\frac1{2\pi}\int_{-\infty}^\infty \varphi(t)\di t=-\tau.
\end{equation*}
We note that by a similar argument, the value of $\mu$ can be determined as $\mu=\tau\sum_{k=1}^{n+1} a_k^3$, although this precise formula will not be used in the rest of the proof.

Finally, by \eqref{eq:lak2}, we have
\begin{equation*}
- \tau a_k^2+\mu a_k=
-\frac1{2\pi}\int_{-\infty}^\infty\varphi(t)\Big(1-\frac{1}{1+\i a_kt}\Big)\di t=-\tau+\frac1{2\pi}\int_{-\infty}^\infty\varphi(t)\frac1{1+\i a_kt}\di t,
\end{equation*}
showing that $a_k$ satisfies the equation \eqref{eq:lms}.
\end{proof}

\medskip
Let us now consider the function
\begin{equation}\label{eq:K}
K(x):=f_{\Sigma_{n+1}+xX}(0)-(- \tau x^2+\mu x+\tau),
\end{equation}
where $\Sigma_{n+1}$ is defined by \eqref{eq:sigmadef}, and $X$ is a random variable exponentially distributed with parameter $1$ and independent of $\{X_j\}_{j=1}^{n+1}$. Based on the probabilistic interpretation, the integral on the right-hand side of \eqref{eq:lms} coincides with the first term in~\eqref{eq:K}. Therefore, any solution of \eqref{eq:lms} belongs to the zero set of $K$. Our goal is to show that there are at most four zeros. Note that by definition of $\tau$, $x=0$ is a trivial zero of $K$.

In the definition of $K(x)$, the variable $x$ appears inside the random variable $\Sigma_{n+1}+xX$, which makes the zeros of $K$ challenging to analyse directly. To overcome this difficulty, we represent $f_{\Sigma_{n+1}+xX}(0)$ in terms of the Laplace transform of a suitable function $\Phi$. Lemma~\ref{lem:LT} then allows us to bound the number of zeros of $K$ by the number of sign changes of $\Phi$. This provides a convenient reduction of the problem. Accordingly, let us introduce 
\begin{equation}\label{eq:Gt}
\Phi(t):=f_{\Sigma_{n+1}}(t)-\Big(\frac{-\tau}2t^2-\mu t+\tau\Big).
\end{equation}

\begin{lemma}[Laplace transform reduction]\label{lem:K-Phi}
For the functions $K$ and $\Phi$ defined in \eqref{eq:K} and \eqref{eq:Gt}, respectively, the following relation holds:
\[\dZ{K}\leq\dZ{\Phi}.\]
\end{lemma}

\begin{proof}

For $x>0$, by \eqref{eq:convolution} and \eqref{eq:exp-dens}, we have
\begin{equation*}\label{eq:LT>0}
\begin{split}
K(x)=
f_{\Sigma_{n+1}+xX}(0)-\frac1x\Big(- \tau x^3+\mu x^2+\tau x\Big)=
\frac1x\int_{0}^\infty \bigg(f_{\Sigma_{n+1}}(-t)-\Big(\frac{- \tau}2t^2+\mu t+\tau\Big)\bigg)\e^{-\frac tx}\di t,
\end{split}
\end{equation*}
while if $x<0$, then by writing $y=-x$, 
\begin{equation*}
K(x)=
f_{\Sigma_{n+1}-yX}(0)-\frac1y\Big(- \tau{y^3}-\mu{y^2}+\tau y\Big)=
\frac1y\int_0^\infty \bigg(f_{\Sigma_{n+1}}(t)-\Big(\frac{- \tau}2t^2-\mu t+\tau\Big)\bigg)\e^{-\frac ty}\di t.
\end{equation*}

Therefore,
\[
K(x)=
\begin{cases}
\dfrac1x\Laplace{\Phi(-\,\cdot)}{\dfrac1x},
& \text{if }x>0,\\[6pt]
0,
& \text{if }x=0,\\[6pt]
-\dfrac1x\Laplace{\Phi}{-\dfrac1x},
& \text{if }x<0.
\end{cases}
\]
Here, in the first line, the Laplace transform is applied to the function $t\mapsto\Phi(-t)$ on $[0,\infty[$, while in the third line, it is applied to the restriction of $\Phi$ to $[0,\infty[$.

Lemma~\ref{lem:LT} bounds the positive and negative zeros of $K$ by the numbers of sign changes of $\Phi$ on $]-\infty,0[$ and $]0,\infty[\,$, respectively. Since $\Phi$ is continuous, the number of sign changes on either open half-line is at most the number of its zeros there. Finally, $K(0)=\Phi(0)$, so a zero at the origin is common to both functions. Hence $\dZ{K}\leq\dZ{\Phi}.$
\end{proof}

\vspace{\medskipamount}
It remains to estimate the number of zeros of $\Phi$. We will argue by induction on the number of random variables, by eliminating the variables $a_jX_j$ from $\Sigma_{n+1}$ one at a time. The convolution structure of exponential random variables ensures that the number of zeros does not increase; hence, the problem reduces to the explicitly tractable case of $\Phi_2$.

\begin{lemma}[Zero-diminishing induction]\label{lem:Phi-induction}
For any  $3 \leq m \leq n+1$ and for any polynomial $q$, there exists a polynomial $\widetilde{q}$ whose leading term is identical to that of $q$, and 
\[
\dZ{f_{\Sigma_m} - q}\leq\dZ{f_{\Sigma_{m-1}}-\widetilde q}.
\]
\end{lemma}

\begin{proof}

Set 
\[
\widetilde q = a_m q' + q,
\]
whose leading term is clearly identical to that of $q$. Then,
\begin{equation}\label{eq:convi}
\widetilde q * f_{a_m X} = q.
\end{equation}
Indeed, assume first that $a_m>0$. By \eqref{eq:exp-dens},
\[
\big(\widetilde q*f_{a_mX}\big)(t)
=\frac1{a_m}\int_{-\infty}^t
\big(a_m q'(s)+q(s)\big)e^{-\frac{t-s}{a_m}}\di s
=e^{-\frac t{a_m}}\int_{-\infty}^t
\frac{\di}{\di s}\big(q(s)e^{\frac s{a_m}}\big)\di s
=q(t)
\]
since $\lim_{s\to-\infty}q(s)e^{\frac s{a_m}}=0$. Similarly, for $a_m<0$, 
\[
\big(\widetilde q*f_{a_mX}\big)(t)
=\frac1{-a_m}\int_t^{\infty}
\big(a_m q'(s)+q(s)\big)e^{-\frac{t-s}{a_m}}\di s
=-e^{-\frac t{a_m}}\int_t^{\infty}
\frac{\di}{\di s}\big(q(s)e^{\frac s{a_m}}\big)\di s
=q(t)
\]
that verifies \eqref{eq:convi}.

Since the convolution is distributive, by \eqref{eq:convolution} and \eqref{eq:convi}, we have
\[
f_{\Sigma_m} - q  =  (f_{\Sigma_{m-1}} -\widetilde q )* f_{a_m X}.
\]
We will compare the number of zeros of the differences on each side.

We may assume that $a_m>0$, since after the reflection $t\mapsto-t$, the argument also applies to the case $a_m<0$. Expanding the convolution above using \eqref{eq:exp-dens},
\[f_{\Sigma_m}(t) - q(t)=\int_{-\infty}^t\frac1{a_{m}}\e^{-\frac{t-s}{a_{m}}}\big(f_{\Sigma_{m-1}}(s) -\widetilde q(s) \big)\di s,
\]
equivalently,
\[\e^{\frac t{a_{m}}}\big(f_{\Sigma_m}(t) - q(t)\big)=\int_{-\infty}^t \frac1{a_{m}}\e^{\frac{s}{a_{m}}}\big(f_{\Sigma_{m-1}}(s) -\widetilde q(s) \big)\di s.\]

Thus, the left-hand side is the cumulative integral of
$
\frac1{a_m}\e^{\frac s{a_m}}
\bigl(f_{\Sigma_{m-1}}(s)-\widetilde q(s)\bigr)
$.
Since ${a_m>0}$, this integrand is absolutely integrable near $-\infty$: the density term is integrable, while the exponential factor dominates the polynomial $\widetilde q$. Hence, 
$\lim_{t\to-\infty}
\e^{\frac t{a_m}}\bigl(f_{\Sigma_m}(t)-q(t)\bigr)=0.
$
Consequently, Lemma~\ref{lem:cumi} gives
\[
\dZ{\e^{\frac t{a_m}}\big(f_{\Sigma_m}(t)-q(t)\big)}
\leq
\dZ{\frac1{a_m}\e^{\frac t{a_m}}
\big(f_{\Sigma_{m-1}}(t)-\widetilde q(t)\big)}.
\]
Since positive factors do not change the number of zeros, this implies that
\[\dZ{f_{\Sigma_m} - q}\leq\dZ{f_{\Sigma_{m-1}} -\widetilde q }.\qedhere\]

\end{proof}

\vspace{\smallskipamount}
\begin{proof}[Proof of Theorem~\ref{th:critic}.]
Let $\mbf a$ be a critical normal vector with nonzero coordinates. By Lemma~\ref{lem:critical-equation}, every coordinate $a_k$ satisfies \eqref{eq:lms}, and hence is a zero of the function $K$ defined in \eqref{eq:K}. Lemma~\ref{lem:K-Phi} gives $\dZ{K}\leq\dZ{\Phi}$. By \eqref{eq:Gt},
\[
\Phi(t)=f_{\Sigma_{n+1}}(t) - p(t)
\]
with
\[
p(t)=\frac{ -\tau}2t^2-\mu t+\tau.
\]

Set $p_{n+1}=p$ and apply Lemma~\ref{lem:Phi-induction} recursively for $m=n+1,n,\ldots,3$ and $q=p_m$, yielding $\widetilde{q} =: p_{m-1}$. Then, quadratic polynomials $p_n, \ldots, p_2$
all have leading term $\frac{- \tau}2t^2$. 

Define $\Phi_{n+1}=\Phi$ and set
\[
\Phi_m(t) = f_{\Sigma_m}(t)-p_m(t)
\]
for each $m=2, \ldots, n$. Then,
\begin{equation}\label{eq:zphin}
\dZ{\Phi} = \dZ{\Phi_{n+1}} \leq \dZ{\Phi_n} \leq \ldots \leq \dZ{\Phi_2}.
\end{equation}
Note that  $p_2$ is a quadratic polynomial with negative leading coefficient, hence it is a concave function. 

It remains to estimate $\dZ{\Phi_2}$. 
Observe that the density function
\begin{equation}\label{eq:2vars}
f_{\Sigma_2}(x)=f_{a_1X_1+a_2X_2}(x)=\begin{cases}\frac1{a_1-a_2}\e^{-\frac x{a_1}},&\text{if}\ x\geq0,\\\frac1{a_1-a_2}\e^{-\frac x{a_2}},&\text{if}\ x\leq0\end{cases}
\end{equation}
is continuous, and its restrictions to the positive and negative half-lines are convex functions. Accordingly, both of these restrictions can intersect $p_2$ at most twice, yielding that 
\[
\dZ{\Phi_2}\leq4.
\]
           
Combined with \eqref{eq:zphin} and Lemma~\ref{lem:K-Phi}, this yields that $K$ has at most four distinct zeros. Since one of these is the trivial root $x=0$, this concludes the proof of Theorem~\ref{th:critic}.
\end{proof}

\begin{figure}[H]
\centering
\includegraphics[scale=.55]{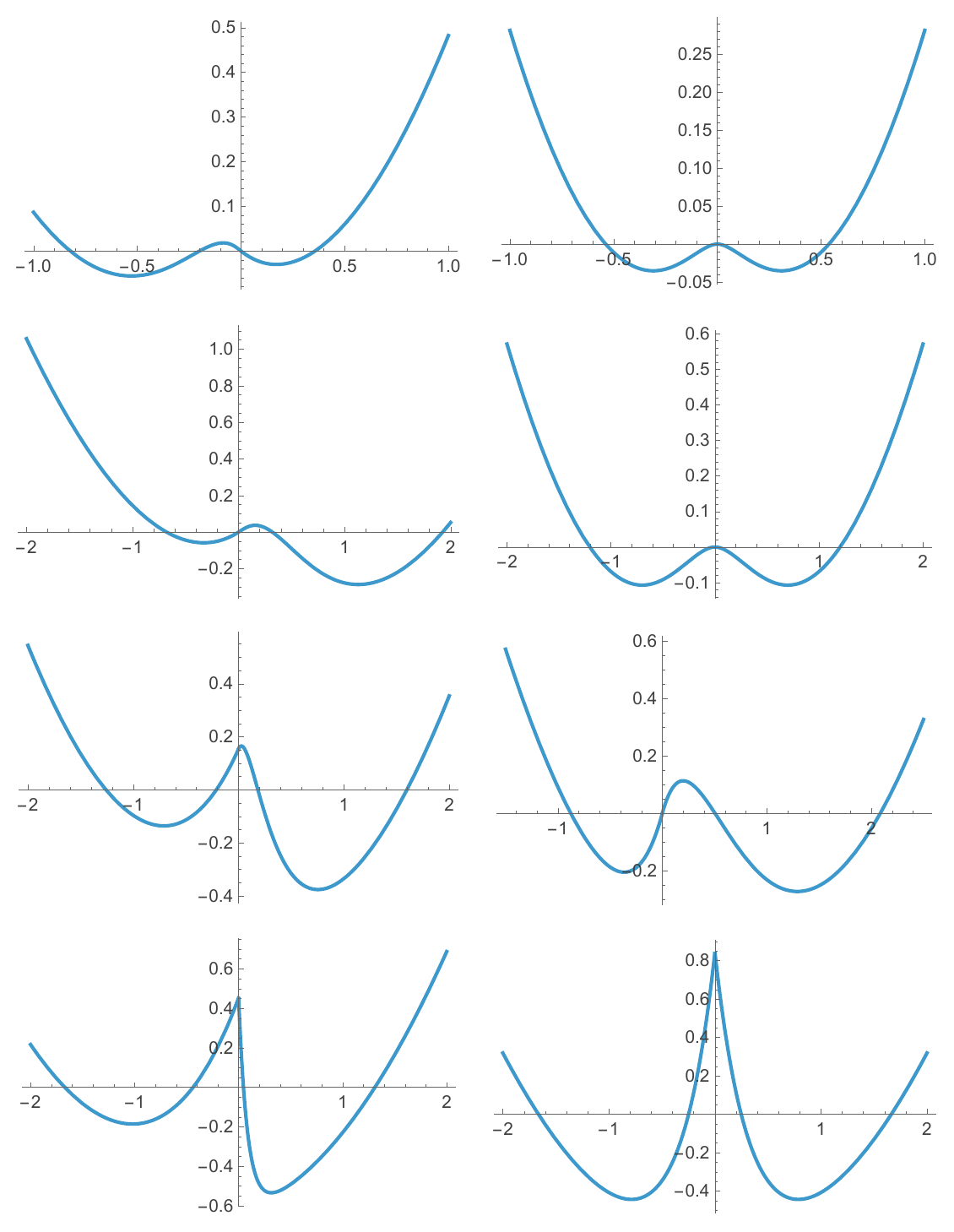}
\caption{Illustration of the argument in two representative examples. The four rows display $K$, $\Phi_4$, $\Phi_3$, and $\Phi_2$, respectively. The first column corresponds to $\mathbf a=\frac1{\sqrt{142}}(1,-10,4,5)$, illustrating the generic case of four distinct zeros. The second column corresponds to $\mathbf a=\frac1{\sqrt{10}}(1,-1,2,-2)$, for which symmetry forces both $K$ and $\Phi_4$ to have only three distinct zeros.}
\end{figure}

\section{Normal vectors with two coordinates}\label{sec:critic2}

In this section, we prove Proposition~\ref{prop:critic2}. Suppose that $\mbf a$ is a normal vector with exactly two distinct coordinates (which must be nonzero). After multiplying $\mbf a$ by $-1$, if necessary, we may write it as
\[
\mbf a_\alpha=(\underbrace{a,\ldots,a}_{\alpha},\underbrace{-b,\ldots,-b}_{\beta}),
\qquad a,b>0,
\]
where $\alpha+\beta=n+1$. The conditions that $\mbf a_\alpha$ be orthogonal to $(1,\ldots,1)$ and have unit norm are
\[
\alpha a=\beta b,
\qquad
\alpha a^2+\beta b^2=1.
\]
Hence,
\[
a=\sqrt{\frac{\beta}{\alpha(n+1)}},
\qquad
b=\sqrt{\frac{\alpha}{\beta(n+1)}},
\]
so the vector is determined by the multiplicity $\alpha$.

\begin{proof}[Proof of Proposition~\ref{prop:critic2}]
By \eqref{eq:Volume-sigma}, it is enough to minimise $\sigma(\mbf a_\alpha)$. As in Webb's work~\cite{Webb}, we evaluate \eqref{eq:G0} by applying the residue theorem in the upper half-plane:
\begin{align*}
\sigma(\mbf a_{\alpha})&=
\frac1{2\pi}\int_{-\infty}^\infty\Big(\frac{1}{1+\i at}\Big)^\alpha\Big(\frac{1}{1-\i bt}\Big)^\beta\di t=
\i\, \Res{\Big(\frac{1}{1+\i at}\Big)^\alpha\Big(\frac{1}{1-\i bt}\Big)^\beta}{\frac\i a}=\\[5pt]&=
\frac\i{(\alpha-1)!}\lim_{t\to\frac\i a}\frac{\di^{\alpha-1}}{\di t^{\alpha-1}}\bigg(\Big(t-\frac\i a\Big)^\alpha\Big(\frac1{1+\i at}\Big)^\alpha\Big(\frac1{1-\i bt}\Big)^\beta\bigg)=\\[5pt]&=
\frac\i{(\alpha-1)!}\cdot(\i a)^{-\alpha}\cdot\lim_{t\to\frac\i a}\bigg(\frac{(n-1)!}{(\beta-1)!}(1-\i bt)^{-n}\cdot(\i b)^{\alpha-1}\bigg)=\\[5pt]&=
\binom{n-1}{\alpha-1}\Big(1+\frac ba\Big)^{-n}\Big(\frac ba\Big)^{\alpha}\frac1b.
\end{align*}
Since $\frac ba=\frac \alpha\beta$ and $b=\sqrt{\frac\alpha{\beta(n+1)}}$, this becomes
\[
\sigma(\mbf a_{\alpha})
=\sqrt{n+1}\binom{n-1}{\alpha-1}
\Big(1+\frac\alpha\beta\Big)^{-n}
\Big(\frac\alpha\beta\Big)^{\alpha-\frac12},
\qquad \beta=n+1-\alpha.
\]
Replacing the normal vector by its negative interchanges $\alpha$ and $\beta$ without changing the section, so we may assume $\alpha\leq\beta$. We claim that $\sigma(\mbf a_\alpha)$ is strictly increasing for $1\leq\alpha\leq\big\lfloor \frac{n+1}2\big\rfloor$.

Indeed, for $1\leq\alpha<\lfloor \frac{n+1}2\rfloor$,
\begin{equation}\label{eq:sigma-ratio}
\frac{\sigma(\mbf a_{\alpha+1})}{\sigma(\mbf a_\alpha)}
=
\frac{(1+\frac1\alpha)^{\alpha+\frac12}}
{(1+\frac1{n-\alpha})^{n-\alpha+\frac12}}.
\end{equation}
Set
\[
h(x)=\Big(1+\frac1x\Big)^{x+\frac12},
\qquad x>0.
\]
Taking the logarithmic derivative, we derive that
\[
\frac{h'(x)}{h(x)}
=\ln\Big(1+\frac1x\Big)-\frac12\Big(\frac1x+\frac1{x+1}\Big).
\]
Since $t\mapsto\frac1t$ is strictly convex,
\[
\ln\Big(1+\frac1x\Big)
=\int_x^{x+1}\frac{1}{t}\di t
<\frac12\Big(\frac1x+\frac1{x+1}\Big),
\]
and hence $h(x)$ is strictly decreasing for $x >0$.

Accordingly, as $\alpha<n-\alpha$ for $\alpha<\big\lfloor \frac{n+1}2\big\rfloor$, by~\eqref{eq:sigma-ratio} we derive that $\sigma(\mbf a_\alpha) < \sigma(\mbf a_{\alpha+1})$ for every $1 \leq \alpha < \big\lfloor \frac{n+1}2\big\rfloor$. Therefore, the minimum of $\sigma(\mbf a_\alpha)$ in the interval under consideration occurs at $\alpha=1$, that is $\beta=n$. Using \eqref{eq:Volume-sigma},
\[
V(\mbf a_\alpha)
\geq \frac{\sqrt{n+1}}{(n-1)!}\,\sigma(\mbf a_1)
=\frac{\sqrt{n+1}}{(n-1)!}\Big(\frac{n}{n+1}\Big)^{n-\frac12}.
\]
The vector $\mbf a_1$ is precisely $\mbf a_{\min}$ from Theorem~\ref{th:simplex}; the case $\beta=1$ gives its negative. This also proves the equality statement.
\end{proof}

\section{Normal vectors with three coordinates values}\label{sec:critic3}

In this section, we prove Proposition~\ref{prop:critic3}. Since $V(\mbf a)=V(-\mbf a)$, after multiplying the normal vector by $-1$ if necessary and permuting its coordinates, we may write
\[
\mbf a=
(\underbrace{a,\ldots,a}_{\alpha},
 \underbrace{-b,\ldots,-b}_{\beta},
 \underbrace{-c,\ldots,-c}_{\gamma}),
\]
where $a>0$, $b\geq c\geq0$, $\alpha+\beta+\gamma=n+1$ and $\alpha,\beta,\gamma>0$. The conditions on $\mbf a$ read as $
\alpha a-\beta b-\gamma c=0$ and 
$
\alpha a^2+\beta b^2+\gamma c^2=1.
$

Set $m=\beta+\gamma$, $u=\frac ba$, $v=\frac ca$, and introduce
\[
x=u-v=\frac{b-c}{a}.
\]
Then, $x\in[0,\frac\alpha\beta]$ and
\begin{equation}\label{eq:uva}
u=\frac{\alpha+\gamma x}{m},\qquad v=\frac{\alpha-\beta x}{m},\qquad{a^2}
=\frac m{\alpha(n+1)+\beta\gamma x^2}.
\end{equation}
The endpoint $x=0$ corresponds to $b=c$, hence to the two-value case with multiplicities $\alpha$ and $m=\beta+\gamma$, while $x=\frac\alpha\beta$ corresponds to $c=0$. We regard $u$, $v$ and $a^2$ as functions of $x$, but omit their argument for simplicity. In particular, $\mbf a=\mbf a(x)$ is uniquely determined by $x$. 
For later reference, we note the formulae
\begin{equation}\label{eq:uvaD}
u'=\frac{\gamma}{m},\qquad
v'=-\frac{\beta}{m},\qquad
(a^2)'=-\frac{2\beta\gamma}{m}xa^4.
\end{equation}

\smallskip

Since $V(\mbf a)$ is a positive constant multiple of $\sigma(\mbf a)$ by \eqref{eq:Volume-sigma}, it is enough to minimise
\[
S(x):=\sigma(\mbf a(x)).
\]
For $c>0$, we evaluate \eqref{eq:G0} by applying the residue theorem in the upper half-plane, exactly as in Section~\ref{sec:critic2}, and use the general Leibniz rule in the third line:
\begin{align*}
S(x)&=
\frac1{2\pi}\int_{-\infty}^\infty\Big(\frac{1}{1+\i at}\Big)^\alpha\Big(\frac{1}{1-\i bt}\Big)^\beta\Big(\frac{1}{1-\i ct}\Big)^\gamma\di t=\\[5pt]&=
\i\, \Res{\Big(\frac{1}{1+\i at}\Big)^\alpha\Big(\frac{1}{1-\i bt}\Big)^\beta\Big(\frac{1}{1-\i ct}\Big)^\gamma}{\frac\i a}=\\[5pt]&=
\frac\i{(\alpha-1)!}\lim_{t\to\frac\i a}\frac{\di^{\alpha-1}}{\di t^{\alpha-1}}\bigg(\Big(t-\frac\i a\Big)^\alpha\Big(\frac1{1+\i at}\Big)^\alpha\Big(\frac1{1-\i bt}\Big)^\beta\Big(\frac{1}{1-\i ct}\Big)^\gamma\bigg)=\\[5pt]&=
\frac\i{(\alpha-1)!}\cdot(\i a)^{-\alpha}\cdot\lim_{t\to\frac\i a}\sum_{k=0}^{\alpha-1}\binom{\alpha-1}k\frac{\di^{k}}{\di t^k}(1-\i bt)^{-\beta}\frac{\di^{\alpha-1-k}}{\di t^{\alpha-1-k}}(1-\i ct)^{-\gamma}=\\[5pt]&=
{a^{-\alpha}}\sum_{k=0}^{\alpha-1}\binom{\beta+k-1}k\Big(1+\frac ba\Big)^{-\beta-k}b^k\binom{\gamma+\alpha-k-2}{\alpha-1-k}\Big(1+\frac ca\Big)^{-\gamma-\alpha+k+1}c^{\alpha-1-k}=\\[5pt]&=
\frac{v^{\alpha-1}}{a(1+u)^{\beta}(1+v)^{\alpha-1+\gamma}}\sum_{k=0}^{\alpha-1}\binom{\beta+k-1}k\binom{\gamma+\alpha-k-2}{\alpha-1-k}\Big(\frac{u}{1+u}\Big)^k\Big(\frac{1+v}{v}\Big)^{k}.
\end{align*}
Write
\[
B_k=\binom{\beta+k-1}{k}\binom{\gamma+\alpha-k-2}{\alpha-1-k}
\]
and
\[
r=r(x)=\frac{u(1+v)}{v(1+u)}.
\]
Then, by setting
\[
Q(r):=\sum_{k=0}^{\alpha-1}B_kr^k,
\]
we obtain
\begin{equation}\label{eq:S-three-values}
S(x)=
\frac{v^{\alpha-1}}{a(1+u)^\beta(1+v)^{\alpha-1+\gamma}}Q(r(x)).
\end{equation}
At $c=0$, equivalently at $x=0$, the value of $S(x)$ is understood by continuity.

\smallskip

The proof is based on studying the polynomial $Q(r)$. A recurrence for the coefficients yields a second-order differential equation for $Q(r)$; after taking a logarithmic derivative, this gives a simple criterion for the second derivative of $S(x)$ at its critical points.

\begin{lemma}\label{lem:second-derivative-three-values}
Assume $c>0$. The function $S(x)$ defined in \eqref{eq:S-three-values} satisfies the following properties:

\begin{enumerate}[label=\emph{\roman*)},leftmargin=*,itemsep=10pt]
\item\label{item:propI} $S'(0)=0$ and $S''(0)>0$.
\item\label{item:propII} At every $x\in\, ]0,\frac{\alpha}{\beta}[$ with $S'(x)=0$, one has
\begin{equation}\label{eq:Ssecond-P}
S''(x)=\Lambda(x)P(x),
\end{equation}
where 
\[
\Lambda(x)=S(x)\frac{\alpha\beta\gamma a^4}{m^4uv}>0
\]
and
\begin{equation}\label{eq:Pdef}
P(x)=2\beta\gamma(\beta-\gamma)x^3-3\beta\gamma(3\alpha+m)x^2-3\alpha(n+1)(\beta-\gamma)x+\alpha(n+1)(\alpha+2m).
\end{equation}
\item\label{item:propIII} If $x\in\, ]0,\frac{\alpha}{\beta}[$ satisfies $S'(x)=P(x)=0$, then
\begin{equation}\label{eq:Sthird-P}
S'''(x)=\Lambda(x)P'(x).
\end{equation}
\end{enumerate}
\end{lemma}

\begin{proof}
For $0\leq k\leq\alpha-2$, the coefficients $B_k$ satisfy
\[
\frac{B_{k+1}}{B_k}
=
\frac{(\beta+k)(\alpha-1-k)}
{(k+1)(\alpha+\gamma-2-k)} \, ,
\]
and hence,
\[
-(k+1)kB_{k+1}+k(k-1)B_k=(\alpha-\beta-2)kB_k-(\alpha+\gamma-2)(k+1)B_{k+1}+(\alpha-1)\beta B_k.
\]
Multiplying by $r^k$ and summing over $k$ gives
\begin{equation}\label{eq:Q-ode}
r(r-1)Q''(r)=
\big((\alpha-\beta-2)r-\alpha-\gamma+2\bigr)Q'(r)
+
(\alpha-1)\beta Q(r),
\end{equation}
where differentiation is taken with respect to $r$. Introduce the logarithmic derivative
\[
M(r)
:=
\frac{\di}{\di\,\ln r}\ln Q(r)
=
\frac{rQ'(r)}{Q(r)}.
\]
This is well defined since $B_k>0$ and $r>0$ imply that $Q(r)>0$.
Differentiating $M$ with respect to $r$ gives
\[
rM'(r)
=
M(r)+\frac{r^2Q''(r)}{Q(r)}-M(r)^2.
\]
For $x>0$, we have $u>v$, and hence $(r-1)>0$.  Thus, from \eqref{eq:Q-ode}, we obtain
\begin{equation}\label{eq:M-ode}
rM'(r)
=\frac{\big((\alpha-\beta-1)r-\alpha-\gamma+1\big)M(r)+(\alpha-1)\beta r}{r-1}-M(r)^2
.
\end{equation}
For $x\in\,[0,\frac\alpha\beta[$, using \eqref{eq:uvaD}, set
\[
\varrho(x):=\big(\ln r(x)\big)'
=\frac1{ma^2uv(1+u)(1+v)}>0.
\]
Taking the logarithmic derivative of
\eqref{eq:S-three-values} gives
\begin{equation}\label{eq:lnS-D}
\big(\ln S(x)\big)'=\varrho(x)D(x)
\end{equation}
with
\[
D(x):=M(r(x))-N(x),
\]
where
\[
N(x)=-\beta\gamma x\,uva^4(1+u)(1+v)-\beta\gamma x\,a^2uv+\beta(\alpha-1)a^2u(1+u)
\]
is defined so that the logarithmic derivative of the prefactor of $Q(r)$ in \eqref{eq:S-three-values} equals $-\varrho(x) N(x)$. 

Differentiating $D$, we can substitute 
\[\frac\di{\di x}M(r(x))=D'(x)+N'(x)\]
into \eqref{eq:M-ode}. This yields the following differential equation for $D$:
\begin{equation}\label{eq:D-ode}
D'(x)
=
C(x) P(x)+B(x)D(x)-\varrho(x)D(x)^2,
\end{equation}
where 
\vspace{-5pt}
\[
C(x)=\frac{\alpha\beta\gamma}{m^3}a^6(1+u)(1+v)>0
\]
is a positive factor, $B(x)$ is a continuous function on $]0,\frac\alpha\beta[$, and $P$ is the polynomial defined in \eqref{eq:Pdef}. The precise form of $B$ will not be needed; its continuity is sufficient for differentiating the term $B(x)D(x)$ at the relevant point. 

We prove the required properties of $S$ using the differential equation \eqref{eq:D-ode}.  For ii), suppose that $S'(x)=0$ for some $x\in\,]0,\frac\alpha\beta[$. Since
$S(x)>0$ and $\varrho(x)>0$, \eqref{eq:lnS-D} implies that $D(x)=0$, and
\eqref{eq:D-ode} gives $D'(x)=C(x)P(x)$. Consequently, by \eqref{eq:lnS-D},
\[
\frac{S''(x)}{S(x)}
=(\ln S(x))''
=\varrho(x)D'(x)
=\varrho(x)C(x)P(x)
=\frac{\alpha\beta\gamma a^4}{m^4uv}P(x),
\]
proving \eqref{eq:Ssecond-P} and thus establishing property ii).

Next, we show property iii). Assume that $S'(x)=P(x)=0$; then, by \eqref{eq:lnS-D} and \eqref{eq:D-ode}, $D(x) = D'(x)=0$. Differentiating
\eqref{eq:D-ode}, at this point gives
$D''(x)=C(x)P'(x)$. Differentiating \eqref{eq:lnS-D} twice and using $D(x)=D'(x)=0$ yields
\[
S'''(x)=S(x)D''(x)=S(x)\varrho(x)C(x)P'(x)
=\Lambda(x)P'(x),
\]
which proves \eqref{eq:Sthird-P} and, thus, iii).

It remains to verify i). Accordingly, let $x=0$. Here, based on~\eqref{eq:uva}, 
$u=v=\frac\alpha m$, $r=1$, and $a^2=\frac m{\alpha(n+1)}$. Evaluating
\eqref{eq:Q-ode} at $r=1$ gives
\vspace{-5pt}
\[
mQ'(1)=(\alpha-1)\beta Q(1),
\]
and hence $M(1)=\frac{\beta(\alpha-1)}m=N(0)$. Thus, $D(0)=0$, so $S'(0)=0$.

To determine the second derivative, differentiate \eqref{eq:Q-ode}
with respect to $r$ and set $r=1$. Then,
\[
(m+1)Q''(1)=(\alpha-2)(\beta+1)Q'(1).
\]
Using this identity together with the formulae~\eqref{eq:uva} for $u,v,a$ and
their derivatives~\eqref{eq:uvaD}, a direct calculation gives
\vspace{-5pt}
\[
D'(0)
=M'(1)-N'(0)
=\frac{\beta\gamma(\alpha+2m)}
{\alpha m^2(m+1)}>0.
\]
Since $\varrho(0)=\frac{m^2}{\alpha(n+1)}>0$, we conclude from
\eqref{eq:lnS-D} that
\[
\frac{S''(0)}{S(0)}
=\varrho(0)D'(0)>0.
\]
Hence, $S''(0)>0$, completing the proof.
\end{proof}

\begin{lemma}\label{lem:P-three-values}
Assume $c>0$. The polynomial $P(x)$ defined by \eqref{eq:Pdef} is strictly concave on $[0,\frac{\alpha}{\beta}]$. Moreover, $P(0)>0$ and
\[
P\left(\frac{\alpha}{\beta}\right)
=-\frac{2\alpha(\alpha-\beta)(\alpha+\beta)m^2}{\beta^2}\,.
\]
Consequently, if $\alpha\leq\beta$ then $P(x)>0$ throughout the open interval $]0,\frac{\alpha}{\beta}[$, while for $\alpha>\beta$, $P(x)$ has a unique zero $x_* \in \,]0,\frac{\alpha}{\beta}[$ with $P(x)>0$ on $]0,x_*[$ and $P(x)<0$ on $]x_*,\frac{\alpha}{\beta}[$.
\end{lemma}

\begin{proof}
From \eqref{eq:Pdef},
\[
P''(x)=6\beta\gamma\bigl(2(\beta-\gamma)x-(3\alpha+m)\bigr).
\]
If $\beta\leq\gamma$, this is negative. If $\beta>\gamma$, then for $x\in[0,\frac\alpha\beta]$, $2(\beta-\gamma)x<2\alpha<3\alpha+m$, so again $P''(x)<0$. Thus, $P(x)$ is strictly concave. The endpoint values are obtained by direct substitution, while the stated sign pattern is an immediate consequence of concavity.
\end{proof}

\begin{proof}[Proof of Proposition~\ref{prop:critic3}]
By Lemma~\ref{lem:second-derivative-three-values},
$S'(0)=0$ and $S''(0)>0$, so $S'(x)>0$ for all sufficiently small
$x>0$. 

By Lemma~\ref{lem:P-three-values}, either $P$ is positive
throughout the admissible interval, or it has a unique zero $x_*$, with
$P(x)>0$ on $]0,x_*[$ and $P(x)<0$ on
$]x_*,\frac\alpha\beta[$.

As long as $P>0$, the function $S'$ cannot have a first zero: at such
a point $x_0$ one would have $S''(x_0)\leq0$, whereas
\eqref{eq:Ssecond-P} gives
\[
S''(x_0)=\Lambda(x_0)P(x_0)>0.
\]
Hence, $S'>0$ throughout the region where $P>0$. In particular, if
$x_*$ exists, then $S'(x)>0$ for $x\in\,]0,x_*[$. Moreover, $S'(x_*)$
cannot vanish, since then $S''(x_*)=0$ and
\eqref{eq:Sthird-P} would give
\[
S'''(x_*)=\Lambda(x_*)P'(x_*)<0,
\]
making $x_*$ a strict local maximum of $S'$ with value $0$, contrary
to the positivity of $S'$ immediately to its left. Thus, $S'(x_*)>0$.

Every interior critical point of $S$ therefore lies in the region
where $P<0$, and by \eqref{eq:Ssecond-P} it is a strict local maximum.
There can be at most one such critical point, since after $S'$ first
crosses from positive to negative, the next return to zero would
have to satisfy $S''\geq0$, whereas \eqref{eq:Ssecond-P} gives
$S''<0$ there. Hence, $S$ has no interior local minimum.

Since $S$ is continuous on $[0,\frac\alpha\beta]$, its minimum is attained
at an endpoint. The endpoint $x=0$ is the two-value case $b=c$, while
$x=\frac\alpha\beta$ is the boundary case $c=0$. This proves the
proposition.
\end{proof}

\section{Completion of the proof}\label{sec:completion}

\begin{proof}[Proof of Theorem~\ref{th:simplex}]
It is enough, by \eqref{eq:Volume-sigma}, to prove
\[
\sigma(\mbf a)\geq \left(\frac{n}{n+1}\right)^{n-\frac12}
\]
for every unit vector $\mbf a \in \R^{n+1}$ with coordinate sum 0. 
We argue by induction on $n$. By a simple calculation based on \eqref{eq:2vars}, the case $n=1$ is immediate. Suppose first that $\mbf a$ has $k\geq1$ zero coordinates, and set $d=n+1-k$. Deleting the zero coordinates gives a unit vector $\mbf b\in\R^d$ with coordinate sum zero and $\sigma(\mbf a)=\sigma(\mbf b)$. By induction,
\[
\sigma(\mbf a)\geq\left(\frac{d-1}{d}\right)^{d-\frac32}
>\left(\frac{n}{n+1}\right)^{n-\frac12},
\]
because $h(t):=\big(\frac{t-1}t\big)^{t-\frac32}$ is strictly decreasing for $t>1$. Indeed,
\[
(\ln h)'(t)=-\ln\left(1+\frac1{t-1}\right)+\frac{t-\frac32}{t(t-1)}<0,
\]
where we used $\ln\big(1+\frac1{t-1}\big)>\frac1t$.

It remains to consider vectors with no zero coordinates. Since the set of admissible unit vectors is compact and $V$ is continuous, a minimiser of $\sigma(\mbf{a})$ exists. Such a vector is a constrained critical point of $\sigma(\mbf a )$, so Theorem~\ref{th:critic} shows that it assumes at most three distinct values. Three distinct values are impossible for a minimiser by Proposition~\ref{prop:critic3}. Therefore, the minimum is attained when $\mbf a$ has two distinct coordinates, and Proposition~\ref{prop:critic2} gives the desired bound. Moreover, its equality statement, together with the strict inequality in the zero-coordinate case, shows that the minimum is attained exactly for the normal vectors obtained from $\mbf a_{\min}$ by permuting the coordinates and multiplying by $-1$.
\end{proof}

\section{Acknowledgements}
The authors are sincerely grateful to Keith Ball for his generous support and encouragement throughout this project. They are also indebted to Hermann König for valuable advice. 

The first named author began this research while at Université Paul Sabatier in Toulouse, working with Franck Barthe, and is grateful for the  welcoming research environment.

\smallskip

Generative artificial intelligence was used in developing the arguments in Section~\ref{sec:critic3}.

\medskip

\footnotesize{Research of G.A. was partially supported by the ERC Advanced Grant "GeoScape" no.  882971, by Hungarian National Research (NKFIH) grants no. KKP-133819, 147145, 147544, and 150151, which has been implemented with the support provided by the Ministry of Culture and Innovation of Hungary from the National Research, Development and Innovation Fund, financed under the ADVANCED-24 funding scheme. 

Work of B.G. was supported by the Warwick Mathematics Institute Centre for Doctoral Training, and gratefully acknowledges funding by University of Warwick’s Chancellors' International Scholarship scheme.

This research was funded by the grant 2024-1.2.8-TÉT-IPARI-CN-2025-00011,
with the support provided by the National Research,
Development and Innovation Office from the National Research,
Development and Innovation Fund, and financed under the
2024-1.2.8-TÉT-IPARI-CN funding scheme.}

\bigskip     

\noindent
{\sc Gergely Ambrus}

\noindent
{\em Bolyai Institute, University of Szeged, Hungary, \\ and HUN-REN Alfréd Rényi Institute of Mathematics,  Budapest, Hungary}

\noindent
e-mail address: \texttt{ambrus@server.math.u-szeged.hu, ambrus@renyi.hu}

\medskip

\noindent
{\sc Barnabás Gárgyán}

\noindent
{\em  Mathematics Institute, University of Warwick, Coventry, United Kingdom, \\ and Bolyai Institute, University of Szeged, Hungary}

\noindent
e-mail address: \texttt{Barnabas.Gargyan@warwick.ac.uk}
\end{document}